\documentclass{amsart}
\usepackage{amsmath,amssymb,amsfonts}
\usepackage{eucal}
\usepackage[normalem]{ulem}
\usepackage{verbatim}

\newtheorem{thm}{Theorem}[section] 
\newtheorem{cor}[thm]{Corollary}

\newtheorem{lem}[thm]{Lemma}

\theoremstyle{definition}

\newtheorem{ex}[thm]{Example}
\newtheorem{rmk}[thm]{Remark}

\theoremstyle{remark}
\newtheorem*{acknowledgments}{Acknowledgments}

\input xy
\xyoption{all}
\newdir{ >}{{}*!/-10pt/@{>}} 

\def\N{\mathbb{N}}

\def\A{\mathcal{A}}
\def\C{\mathfrak{C}}
\def\D{\Delta}
\def\L{\Lambda}

\def\a{\alpha}
\def\b{\beta}

\def\e{\epsilon}

\def\s{\sigma}
\def\t{\tau}

\def\sk{\text{sk}}
\def\cosk{\text{cosk}}
\def\coeq{\text{coeq}}

\def\lim{\text{lim}}
\def\colim{\text{colim}}

\def\ra{\rightarrow}

\def\Del{{\bf \Delta}}
\def\Cat{\text{\bf Cat}}

\def\Top{\text{\bf Top}}
\def\sSet{\text{\bf sSet}}
\def\sCat{\text{\bf sCat}}

\begin{document}

\title[Simplicial categories associated to quasi-categories]{On the structure of simplicial categories associated to quasi-categories}

\author[Emily Riehl]{EMILY RIEHL\\ Department of Mathematics, University of Chicago\\ 5734 S.~University Ave., Chicago, IL 60637 \\ e-mail\textup{: \texttt{eriehl@math.uchicago.edu}}}

\date{December 1, 2010}

\begin{abstract} The homotopy coherent nerve from simplicial categories to simplicial sets and its left adjoint $\C$ are important to the study of $(\infty,1)$-categories because they provide a means for comparing two models of their respective homotopy theories, giving a Quillen equivalence between the model structures for quasi-categories and simplicial categories. The functor $\C$ also gives a cofibrant replacement for ordinary categories, regarded as trivial simplicial categories. However, the hom-spaces of the simplicial category $\C X$ arising from a quasi-category $X$ are not well understood. We show that when $X$ is a quasi-category, all $\L^2_1$ horns in the hom-spaces of its simplicial category can be filled. We prove, unexpectedly, that for any simplicial set $X$, the hom-spaces of $\C X$ are 3-coskeletal. We characterize the quasi-categories whose simplicial categories are locally quasi, finding explicit examples of 3-dimensional horns that cannot be filled in all other cases. Finally, we show that when $X$ is the nerve of an ordinary category, $\C X$ is isomorphic to the simplicial category obtained from the standard free simplicial resolution, showing that the two known cofibrant ``simplicial thickenings'' of ordinary categories coincide, and furthermore its hom-spaces are 2-coskeletal.
\end{abstract}

\maketitle

\section{Introduction}\label{intsec}

In recent years, many advances have been made in the study of $(\infty,1)$-\emph{categories}, loosely defined to be categories enriched in $\infty$-groupoids or spaces. Models of $(\infty,1)$-categories abound, but in this paper we restrict our attention to two of the simplest:  quasi-categories, which are simplicial sets that satisfy a particular horn-filling property, and simplicially enriched categories (henceforth, \emph{simplicial categories}). The categories $\sSet$ and $\sCat$ each bear a model structure such that the fibrant objects are the models of $(\infty,1)$-categories --- in the latter case, the simplicial categories whose hom-spaces are Kan complexes. Furthermore, there is a Quillen equivalence between them, with the right adjoint the homotopy coherent nerve $\N \colon \sCat \ra \sSet$ of  Jean-Marc Cordier \cite{cordierdiagramme} (cf. the survey article \cite{bergnersurvey} or \cite[chapters 1 and 2]{lurietopos}).

This adjunction provides the primary means of translating between these two models, which accounts for its importance. Interesting examples of quasi-categories are often presented as homotopy coherent nerves of fibrant simplicial categories: the ``quasi-category of spaces'' is one example. More exotically, any locally presentable quasi-category is equivalent to the homotopy coherent nerve of a combinatorial simplicial model category \cite[proposition A.3.7.6]{lurietopos}. Conversely, the left adjoint $\C\colon \sSet \ra \sCat$ ``rigidifies'' a simplicial set $X$, in which one can define a natural notion of simplicial ``hom-space'' between two vertices, into a category whose simplicial enrichment is given strictly. The vertices of the hom-spaces of $\C X$ are ``composable'' paths of edges in $X$, meaning that the target of each edge in the sequence is the source of the next, not necessarily that any ``composite'' edge exists. If $X$ is a quasi-category, then these edges can be composed, and the 1-simplices in the hom-spaces correspond to higher simplices in $X$, which ``witness'' the sense in which a particular edge is a composite of a path given by others. It is tempting to describe these 1-simplices as homotopies between the various paths, but this isn't an exact analogy because the 1-simplices are directed: the source is always a shorter sequence of edges than the target. Instead, we prefer to think of the 1-simplices as ``factorisations,'' which exhibit how one path of edges can be deformed into a longer one. The higher simplices of the hom-spaces of $\C X$ are ``higher homotopies'' that exhibit coherence relations between the ``factorisations'' relating the various paths.

By a general categorical principle \cite[proposition 3.1.5]{hoveymodel}, the homotopy coherent nerve and its left adjoint $\C$ are determined by a cosimplicial object in $\sCat$, that is, a functor $\C \D^-\colon \Del \ra \sCat$, where $\Del$ is the usual category of finite non-empty ordinals $[n] = \{0,\ldots, n\}$. The homotopy coherent nerve of a simplicial category $\mathcal{C}$ is the simplicial set with $n$-simplices the simplicial functors $\C\D^n \ra \mathcal{C}$, and the functor $\C$ is the left Kan extension of this functor along the Yoneda embedding $\Del \ra \sSet$. This can be computed by a a familiar coend formula, which we describe in the next section. It follows that $\C X$ is the simplicial category ``freely generated'' by $X$, in the sense that $n$-simplices in $X$ correspond to simplicial functors $\C \D^n \ra \C X$, which, we shall see below, should be thought of as homotopy coherent diagrams in $\C X$. Thus, to understand the adjunction $\C\dashv \N$, we must first build intuition for the $\C \D^n$.

There are many ways to describe the simplicial categories $\C \D^n$, one of which employs the free simplicial resolution construction of \cite{cordierdiagramme}, \cite{dwyerkansimplicial}, and elsewhere. By repeatedly applying the free category comonad on $\Cat$ induced by the free-forgetful adjunction $F \dashv U$ between categories and reflexive, directed graphs to the poset category $[n]$, one obtains a simplicial object in $\Cat$:
\[\xymatrix{FU[n] \ar[rr]|-{F\eta U} & & FUFU[n] \ar@<2ex>[ll]|-{FU\epsilon} \ar@<-
2ex>[ll]|-{\epsilon FU} \ar@<-2ex>[rr]|-{FUF\eta U} \ar@<2ex>[rr]|-{F\eta UFU} 
& & \ar[ll]|-{FU\epsilon FU} \ar@<-4ex>
[ll]|-{\epsilon FUFU} \ar@<4ex>[ll]|-{FUFU\epsilon} FUFUFU[n] & \cdots}\]
Each of these categories has the same objects as $[n]$. The arrows of $FU[n]$ are sequences of composable non-identity morphisms in $[n]$. The arrows of $FUFU[n]$ are again such sequences but with every morphism appearing in exactly one set of parentheses. The arrows of $FUFUFU[n]$ are sequences of composable non-identity morphisms with every morphism appearing in exactly two sets of parentheses, and so forth. The face maps $(FU)^k\e(FU)^j$ remove the parentheses that are contained in exactly $k$ others; $FU\cdots FU\e$ composes the morphisms inside the innermost parentheses. The degeneracy maps $F(UF)^k\eta(UF)^jU$ double up the parentheses that are contained in exactly $k$ others; $F\cdots UF\eta U$ inserts parentheses around each individual morphism.

A simplicial object in $\Cat$ determines a simplicial category exactly when each of the constituent functors acts as the identity on objects, as is the case here. Hence, this construction specifies a simplicial category,  which is $\C \D^n$. The vertices of the hom-spaces are the sequences of composable non-identity morphisms, i.e., the arrows of $FU[n]$; the 1-simplices of the hom-spaces are the arrows of $FUFU[n]$; and so forth.

More geometrically, $\C\D^n$ is the simplicial category with objects $0,\ldots, n$ and hom-spaces $\C\D^n(i,j)$ defined to be (ordinary) nerves of certain posets $P_{i,j}$. If $j<i$, then $P_{i,j}$ and hence $\C\D^n(i,j)$ is empty. Otherwise $P_{i,j}$ is the poset of subsets of the interval $\{k \mid i \leq k \leq j \} \subset [n]$ that contain both endpoints. For $j=i$ and $j=i+1$, this poset is the terminal category. For $j > i+1$, a quick calculation shows that $P_{i,j}$ is isomorphic to the product of the category $[1]$ with itself $j-i-1$ times. Hence \[\C\D^n(i,j) = \begin{cases} (\D^1)^{j-i-1} & \text{when}\ j > i, \\ \D^0 & \text{when}\ j=i,\\ \emptyset & \text{when}\  j < i. \end{cases}\] For proof that these two descriptions coincide, see \cite[\S 2]{duggerspivakrigid}. 

Here is some intuition for these definitions. The hom-space $\C\D^n(i,j)$ parametrises paths from $i$ to $j$ in the poset $[n]$. The vertices count the number of distinct paths: if $j = i+2$, then there are two options --- one which passes through the object $i+1$ and one which avoids it --- and, accordingly, the simplicial set $\C\D^n(i,i+2) = \D^1$ has two vertices. The higher dimensional data is designed such that \emph{homotopy coherent} diagrams $[n] \ra \mathcal{C}$, studied extensively by Cordier and Timothy Porter (but see also \cite[section 1.2.6]{lurietopos}), correspond to simplicial functors $\C\D^n \ra \mathcal{C}$. We illustrate the case where $n=3$ and $\mathcal{C} = \Top_{\sSet}$, the simplicially enriched category of compactly generated spaces, with simplicial enrichment given by applying the total singular complex functor to each hom-space (cf. \cite[section 1]{cordierporterhomotopy}). A homotopy \emph{commutative} diagram in $\Top_{\sSet}$ picks out spaces $X$, $Y$, $Z$, and $W$ and functions \[\xymatrix{ X \ar[r]^f \ar@/_3ex/[rr]_j \ar@/_7ex/[rrr]_l & Y \ar[r]^g \ar@/_3ex/[rr]_k & Z \ar[r]^h & W}\] such that there exist homotopies $j \simeq gf$, $k \simeq hg$, $l \simeq kf$, $l \simeq hj$, and $l \simeq hgf$. Such a diagram is homotopy \emph{coherent} if one can chose the above homotopies in such a way that the composite homotopies $l \simeq hgf$ are homotopic in the sense illustrated below \[\xymatrix@=15pt{ l \ar[rr] \ar[dd] \ar[ddrr] \ar@{}[drr]|(.67){\simeq} \ar@{}[ddr]|(.67){\simeq} & & kf \ar[dd] \\ & & \\ hj \ar[rr] & & hgf}\] The data specifying these each of these homotopies is precisely the image of the simplicial map $\D^1 \times \D^1 = \C\D^3(0,3) \ra \Top_{\sSet}(X,W)$.

A more modern treatment of these same ideas is given in the straightening construction of \cite[chapters 2 and 3]{lurietopos}, which associates contravariant simplicial (or marked simplicial) functors with domain $\C X$ to right fibrations (Cartesian fibrations) over $X$, which can be thought of as contravariant Kan-complex-valued (quasi-category-valued) pseudofunctors. The functor $\C$ figures prominently in this correspondence. The work contained in this paper was motivated by our attendant desire to be able to compute particular examples of this construction.

In the next section, we unravel the definition of the simplicial category $\C X$ associated to a simplicial set $X$ and describe the lower dimensional simplices of the hom-spaces $\C X(x,y)$. The intuition provided by this calculation is satisfyingly confirmed by recent work of Dugger and Spivak \cite{duggerspivakrigid}, which identifies the $n$-simplices of $\C X(x,y)$ with necklaces in $X$, accompanied by certain vertex data. In Section \ref{21sec}, we use this characterization to prove that all $\L^2_1$ horns in $\C X(x,y)$ can be filled, when $X$ is a quasi-category. In Section \ref{cosksec}, we demonstrate that the necklace representation is even more useful in higher dimensions, proving the surprising fact that for any simplicial set $X$, $\C X(x,y)$ is 3-coskeletal, which says that sufficiently high dimensional simplicial spheres in these hom-spaces can be filled uniquely. 

In light of these results, one might hope that the simplicial category associated to a quasi-category is locally quasi; however, this is seldom the case. In Section \ref{charsec}, we show that if $\C X$ is locally quasi, then $X$ is the (ordinary) nerve of a category. 
We consider the case when $X$ is the nerve of a category in Section \ref{catsec}; $\C X$ is then its cofibrant replacement. We prove that the hom-spaces of $\C X$ are 2-coskeletal, but that for most categories, there are $\L^3_1$ and $\L^3_2$ horns in certain hom-spaces that cannot be filled. Finally, we show that the simplicial category obtained by applying the free simplicial resolution construction described above to any small category $\A$ is isomorphic to the simplicial category $\C N\A$, where $N \colon \Cat \ra \sSet$ is the ordinary nerve functor. In other words, these constructions coincide for all categories, not just for the poset categories $[n]$.

\section{Understanding the hom-spaces $\C X(x,y)$}

Although we are most interested in understanding the simplicial category associated to a quasi-category, all of the results in this section apply for a generic simplicial set $X$. The notation for simplicial sets throughout this paper is consistent with \cite{goerssjardinesimplicial}. By definition \[\C X = \int^{[n] \in \Del} \coprod_{X_n}\, \C \D^n = \coeq \left( \coprod_{f \colon [m] \ra [n]}\, \coprod_{X_n} \C \D^{m} \xymatrix{ \ar@<.5ex>[r] \ar@<-.5ex>[r] &} \coprod_{[n]} \coprod_{X_n} \C \D^n \right).\]
It suffices to restrict the interior coproducts to the non-degenerate simplices of $X$ and the left outer coproduct to the generating coface maps $d^i \colon [n-1] \ra [n]$. Write $\tilde{X}_n$ for the non-degenerate $n$-simplices of $X$. Then
\[\C X = \colim\left( \raisebox{.6in}{ \xymatrix@R=10pt{ & & \coprod_{\tilde{X}_1} \C\D^0  \ar@<.5ex>[dd]^{d^1} \ar@<-.5ex>[dd]_{d^0} \ar@<.5ex>[r]^{d_0} \ar@<-.5ex>[r]_{d_1} & \coprod_{X_0} \C\D^0 \\ & &  \\ & \coprod_{\tilde{X}_2} \C\D^1 \ar@<1ex>[r]\ar[r] \ar@<-1ex>[r] \ar@<1ex>[dd] \ar[dd] \ar@<-1ex>[dd] & \coprod_{\tilde{X}_1} \C\D^1 & \\ & & \\ & \ar@{}[l]|-{\cdot\,\cdot\,\cdot} \coprod_{\tilde{X}_2} \C\D^2}}\right) \footnote{Some non-degenerate $n$-simplices may have degenerate $n-1$-simplices as faces, so we cannot technically restrict the face maps to maps $d_i \colon \tilde{X}_n \ra \tilde{X}_{n-1}$. Instead, one must attach each degenerate face to the unique lower-dimensional non-degenerate simplex it represents, but this technicality will not affect our conceptual discussion.}\]

The objects of $\C X$ are the vertices of $X$.  The simplicial categories $\C\D^0$ and $\C\D^1$ are the free simplicial categories on the poset categories $[0]$ and $[1]$ respectively, and the free simplicial category functor is a left adjoint and so commutes with colimits. Hence, if $X$ is 1-skeletal so that $\tilde{X}_n = \emptyset$ for all $n>1$, then $\C X$ is the free simplicial category on the graph with vertex set $X_0$ and edge set $\tilde{X}_1$. Concretely, this means that the hom-spaces $\C X(x,y)$ are all discrete simplicial sets containing a vertex for each path of edges from $x$ to $y$ in $X$.

In general, for each 2-simplex of $X$ with boundary as shown 
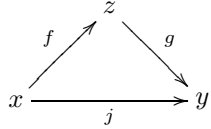
\[\xymatrix{ & z \ar[dr]^g & \\ x \ar[ur]^f \ar[rr]_j & & y}\] there exists a 1-simplex from the vertex $j$ to the vertex $gf$ in $\C X(x,y)$. Furthermore, for each vertex in some hom-space representing a sequence of paths containing $j$, there is a 1-simplex connecting it to the vertex representing the same sequence, except with $gf$ in place of $j$. 

However, the 2-skeleton of $X$ does not determine the 1-skeleta of the hom-spaces. For example, for each 3-simplex $\s$ of $X$ as depicted below 
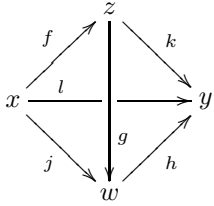
\[\xymatrix{ & z \ar[dd]^(.7)g \ar[dr]^k & \\ x \ar[ur]^f \ar[dr]_j \ar'[r]^l[rr] & & y \\ & w \ar[ur]_h &}\] there is an edge from $l$ to $hgf$ in $\C X(x,y)$. In general, there is an edge between the vertices represented by paths $p_1\ldots p_r$ and $q_1 \ldots q_r$ of edges from $x$ to $y$ in $X$ if and only if each edge $p$ in the first path that does not appear in the second is replaced by a sequence of $n$-edges that appear as the \emph{spine} of some $n$-simplex of $X$ with $p$ as its \emph{diagonal}. Here, the \emph{spine} of an $n$-simplex is the sequence of edges between the adjacent vertices, using the usual ordering of the vertices, and the \emph{diagonal} is the edge $[0,n]$ from the initial vertex to the final one.

In this way, each edge of $\C X(x,y)$ corresponds to a \emph{necklace} \[\D^{n_1} \vee \cdots \vee \D^{n_r} \ra X\] in $X$. By $\D^n \vee \D^k$ we always mean that the final vertex of the $n$-simplex is identified with the initial vertex of the $k$-simplex. A necklace is comprised of a sequence of \emph{beads}, the $\D^{n_i}$ above,  that are strung together along the \emph{joins}, defined to be the union of the initial and final vertices of each bead. When speaking colloquially, we may associate a bead of the necklace with its image, a simplex of the appropriate dimension in $X$, and the faces of the bead with the corresponding faces of the simplex.

By a theorem of Daniel Dugger and David Spivak, necklaces can be used to characterize the higher dimensional simplices of the hom-spaces $\C X(x,y)$ as well, provided we keep track of additional vertex data. 

\begin{thm}[Dugger, Spivak {\cite[corollary 4.8]{duggerspivakrigid}}]\label{dsthm} Let $X$ be a simplicial set with vertices $x$ and $y$. An $n$-simplex in $\C X(x,y)$ is uniquely represented by a triple $(T, f, \vec{T})$,  where $T$ is a necklace; $f \colon T \ra X$ is a map of simplicial sets that sends each bead of $T$ to a non-degenerate simplex of $X$ and the endpoints of the necklace to the vertices $x$ and $y$, respectively; and $\vec{T}$ is a flag of sets  \[J_T = T^0 \subset T^1 \subset T^2 \subset \cdots \subset T^{n-1} \subset T^n= V_T\] of vertices $V_T$ of $T$, where $J_T$ is the set of joins of $T$.
\end{thm}

Necklaces $f\colon T \ra X$ with the property described above are called \emph{totally non-degenerate}. 
Note that the map $f$ need not be injective. If $x = y$ is a vertex with a non-degenerate edge $e\colon x \ra y$, the map $e\colon \D^1 \ra X$ defines a totally non-degenerate necklace in $X$.

Dugger and Spivak prefer to characterize the simplices of $\C X(x,y)$ as equivalence classes of triples $(T,f,\vec{T})$, which are not necessarily totally non-degenerate \cite[corollary 4.4, lemma 4.5]{duggerspivakrigid}. However, it is always possible to replace an arbitrary triple $(T,f,\vec{T})$ by its unique totally non-degenerate quotient.

\begin{lem}[Dugger, Spivak {\cite[proposition 4.7]{duggerspivakrigid}}]\label{tndlem} Let $X$ be a simplicial set and suppose $T$ is a necklace; $f \colon T \ra X$ is a map of simplicial sets; and $\vec{T}$ is a flag of sets  \[J_T = T^0 \subset T^1 \subset T^2 \subset \cdots \subset T^{n-1} \subset T^n= V_T\] of vertices $V_T$ of $T$, where $J_T$ is the set of joins of $T$. Then there is a unique quotient $(\overline{T}, \overline{f}, \overline{\vec{T}})$ of this triple such that $f$ factors through $\overline{f}$ via a surjection $T \twoheadrightarrow \overline{T}$ and $(\overline{T}, \overline{f}, \overline{\vec{T}})$ is totally non-degenerate.
\end{lem}
\begin{proof}
By the Eilenberg-Zilber lemma \cite[proposition II.3.1, pp. 26-27]{gabrielzismancalculus}, any simplex $\s \in X_n$ can be written uniquely as $\e \s'$ where $\s' \in X_m$ is non-degenerate, with $m \leq n$, and $\e \colon X_m \ra X_n$ is a simplicial operator corresponding to a surjection $[n] \ra [m]$ in $\Del$. The necklace $\overline{T}$ agrees with $T$ at each bead whose image is non-degenerate. If $\s$ is a degenerate $n$-simplex in the image of $f$, then to form $\overline{T}$ we replace the bead $\D^n$ of $T$ corresponding to $\s$ by the bead $\D^m$, where $m$ is determined by the Eilenberg-Zilber decomposition of $\s$, described above. Define $\overline{f} \colon \overline{T} \ra X$ restricted to this $\D^m$ to equal $\s'$. The morphism $\e$ defines a surjective map of simplicial sets $\D^n \ra \D^m$, which defines the quotient map $T \twoheadrightarrow \overline{T}$ at this bead in such a way that $f$ factors through $\overline{f}$ along this map.

Let $\overline{\vec{T}}$ be the flag of sets of vertices of $\overline{T}$ given by the direct image of $\vec{T}$ under $T \twoheadrightarrow \overline{T}$. The resulting triple $(\overline{T}, \overline{f}, \overline{\vec{T}})$ is totally non-degenerate and unique such that $f$ factors through $\overline{f}$.
\end{proof}

With the aid of Lemma \ref{tndlem}, the face maps $d_i \colon \C X(x,y)_n \ra \C X(x,y)_{n-1}$ can also be described in the language of flags and necklaces. In what follows, $d_i \vec{T}$ denotes the flag of sets $$J_T  = T^0 \subset \cdots \subset \widehat{T^i} \subset \cdots \subset T^n = V_T$$ with $T^i$ removed from the sequence.

\begin{thm}[Dugger, Spivak {\cite[remarks 4.6 and 4.9]{duggerspivakrigid}}]\label{dsrmk}  The faces of an $n$-simplex $(T,f, \vec{T})$ are uniquely represented by the triples \[d_i (T,f, \vec{T}) = \begin{cases} (\overline{T^*},\overline{f\vert_{T^*}}, \overline{d_0\vec{T}}) & i=0 \\ (T,f , d_i\vec{T}) & 0 < i < n \\ (\overline{T'}, \overline{f\vert_{T'}}, \overline{d_n\vec{T}}) & i = n \end{cases}\] where $T' = T\vert_{T^{n-1}}$ is the maximal subnecklace of $T$ with vertices $T^{n-1}$ and $T^*$ is the maximal subnecklace of $T$ with joins $T^1$. If the triples $(T^*, f\vert_{T^*}, d_0\vec{T})$ and $(T', f\vert_{T'}, d_n\vec{T})$ are totally non-degenerate, then these are the zeroth and $n$-th faces, respectively; otherwise, these faces are given by the unique quotients of Lemma \ref{tndlem}.
\end{thm}

We call $T^*$ the $T^1$-\emph{splitting} of $T$. Each bead of $T$ is replaced by a necklace with the same spine whose beads are each faces of the original bead. The vertices of each new bead will be a consecutive subset of vertices of the bead of $T$ with initial and final vertices in $T^1$. The sum of the dimensions of these new beads will equal the dimension of the original bead.

\begin{ex} We use Theorem \ref{dsrmk} to compute the faces of the 3-simplex \[(\D^6, \s \colon \D^6 \ra X, \{0,6\} \subset \{0,3,4,6\} \subset \{0,1,3,4,6\} \subset [6]).\] 
The necklace $\D^3 \vee \D^1 \vee \D^2$  
is the $\{0, 3, 4, 6\}$-splitting of the necklace $\D^6$, so the zeroth face is \[(\D^3 \vee \D^1 \vee \D^2, \s\vert_{\sim} \colon \D^3 \vee \D^1 \vee \D^2 \ra X, \{0, 3,4,6\} \subset \{0,1,3,4,6\} \subset [6])\] where the map $\s\vert_{\sim}$ is given on the first bead by restricting $\s$ to the face containing the vertices $0$, $1$, $2$, and $3$; on the second bead by restricting to the face containing the vertices $3$ and $4$; and on the third bead by restricting to the face containing the vertices $4$, $5$, and $6$. The first and second faces are 
\begin{align*}  (\D^6, \s \colon \D^6 \ra X, \{0,6\} \subset \{0,1,3,4,6\} \subset [6])& \\  \text{and} \quad (\D^6, \s \colon \D^6 \ra X, \{0,6\} \subset \{0,3,4,6\} \subset [6])&.\\ \intertext{Note that the necklace parts of these faces are the same. The third face is} (\D^4, d_2d_5\s \colon \D^4 \ra X, \{0,4\} \subset \{0,2,3,4\} \subset [4])&\end{align*} where we have chosen to rename the vertices in the flag. The simplicial map $d_2d_5\s \colon \D^4 \ra X$ is the restriction of $\s$ to the face containing all vertices except for 2 and 5.
\end{ex}

We note an obvious corollary to Theorem \ref{dsrmk}, which will be frequently exploited.

\begin{cor}\label{dscor} The necklaces representing the \emph{inner} faces of an $n$-simplex $(T,f,\vec{T})$ are all equal to $T$.
\end{cor}
\begin{proof} If $0 < i < n$, the triple $(T,f,d_i\vec{T})$ is totally non-degenerate and represents the $i$-th face of $(T,f,\vec{T})$.
\end{proof}

When describing $n$-simplices of $\C X(x,y)$ as triples $(T,f, \vec{T})$, we frequently define the necklace as a subsimplicial set of $X$, in which case the map $f$ is understood and may be omitted from our notation. In dimensions 0 and 1, the flag $\vec{T}$ is completely determined by the necklace $T$, so we may represent the simplices of $\C X(x,y)$ by necklaces alone.

\section{Filling $\L^2_1$ horns in $\C X(x,y)$}\label{21sec}

Recall, a \emph{quasi-category} is a simplicial set $X$ such that every \emph{inner horn} in $X$ has a \emph{filler}, i.e., any diagram $$\xymatrix{ \L^n_k \ar[r] \ar[d] & X \\ \D^n \ar@{-->}[ur] &}$$ has the indicated extension. Here, $\L^n_k$ is the simplicial subset of the standard (represented) $n$-simplex $\D^n$ generated by all of the $(n-1)$-dimensional faces except for the $k$-th, where $0 < k < n$. 

Horn filling conditions in quasi-categories guarantee that simplices in each dimension can be composed. In this section, we will use the characterization of simplices in the hom-spaces $\C X(x,y)$ as necklaces to prove the following theorem, which says that the ``factorisations'' relating paths of edges in a quasi-category can be composed.

\begin{thm}\label{21thm} Let $X$ be a quasi-category and let $x$ and $y$ be any two vertices. Then every horn $\L^2_1 \ra \C X(x,y)$ has a filler $\D^2 \ra \C X(x,y)$.
\end{thm}

Our proof will use a lemma due to Andr\'{e} Joyal describing maps constructed from \emph{joins} of simplicial sets, where the join, denoted $\star$, is the restriction of the Day tensor product \cite{dayonclosed} on augmented simplicial sets arising from the monoidal structure on $\Del_{+}$, the category of all finite ordinals and order preserving maps. The reader who wishes to verify the proofs of Corollaries \ref{cor1} and \ref{cor2} below should see \cite{joyalquasiandkan} for an explicit definition.

Recall a monomorphism of simplicial sets is \emph{mid anodyne} if it is in the saturated class generated by the inner horn inclusions $\L^n_k \ra \D^n$ and \emph{left anodyne} if it is in the saturated class generated by the horn inclusions with $0 \leq k < n$. The left anodyne maps are precisely those maps which have the left lifting property with respect to the \emph{left fibrations}, which are the maps that have the right lifting property with respect to the appropriate horn inclusions.

\begin{lem}[Joyal {\cite[lemma 2.1.2.3]{lurietopos}}]\label{joinlem} If $u \colon X \ra Y$ and $v \colon Z \ra W$ are monomorphisms of simplicial sets such that $v$ is left anodyne, then the map \[u \hat{\star} v \colon X \star W \cup_{X \star Z} Y \star Z \ra Y \star W\] is mid anodyne. 
\end{lem}

We will apply Lemma \ref{joinlem} in the case where $v$ is the inclusion $i_0\colon \D^0 \ra \D^n$ of the initial vertex of $\D^n$.

\begin{lem} The map $i_0 \colon \D^0 \ra \D^n$ is left anodyne.
\end{lem}
\begin{proof}
It suffices to show that $i_0$ lifts against any left fibration. Given a left fibration $p \colon X \ra Y$ and a lifting problem
\[\xymatrix{ \D^0 \ar[d]_{i_0} \ar[r]^a & X \ar[d]^p \\ \D^n \ar[r]_b & Y}\] we may lift the edges $[0,k]$ of the $n$-simplex $b\colon \D^n \ra Y$ to $X$ because we can solve lifting problems of the form \[\xymatrix{ \D^0 = \L^1_0 \ar[r]^-a \ar[d] & X \ar[d]^p \\ \D^1 \ar[r]_-{b \cdot i_{[0,k]}} \ar@{-->}[ur] & Y}\] This allows us to lift all 2-dimensional faces of $b$ containing the vertex 0 by filling $\L^2_0$ horns. In turn, this allows us to lift all 3-dimensional faces of $b$ containing the vertex 0 by filling $\L^3_0$ horns, and inductively we can lift all ($n-1$)-dimensional faces of $b$ containing 0 by filling $\L^{n-1}_0$ horns. These faces form a $\L^n_0$ horn in $X$ whose image under $p$ is the corresponding horn of $b$ in $Y$. We lift against the inclusion $\L^n_0 \ra \D^n$ to obtain the desired lift of our $n$-simplex $b$.
\end{proof}

\begin{cor}\label{cor1} The inclusion $\D^k \vee \D^n \ra \D^{k+n}$ is mid anodyne.
\end{cor}
\begin{proof} Apply Lemma \ref{joinlem} with $u \colon \emptyset \ra \D^{k-1}$ and $v = i_0 \colon \D^0 \ra \D^n$.
\end{proof}

\begin{cor}\label{cor2} The inclusion $\D^k \cup_{[0,k]=[0,1]} \D^n \ra \D^{k+n-1}$ is mid anodyne.
\end{cor}
\begin{proof} Apply Lemma \ref{joinlem} with $u= i_0 \colon \D^0 \ra \D^{k-1}$ and $v= i_0 \colon \D^0 \ra \D^{n-1}$.
\end{proof}

We now have all of the tools necessary to prove Theorem \ref{21thm}.

\begin{proof}[Proof of Theorem \ref{21thm}]
By Theorem \ref{dsthm}, a horn $\L^2_1 \ra \C X(x,y)$ is specified by necklaces $T$ and $U$ such that $d_0T = d_1U$. We use Theorem \ref{dsrmk} to compute these faces. 

Let $U'$ be  the subnecklace of $U$ whose vertices equal the joins $J_U$ of $U$; we call $U'$ the \emph{diagonal} of $U$. The necklace $U'$ is comprised of the path of edges which form the diagonal edges of the beads of $U$; for each bead $\D^k$ of $U$, the $[0,k]$ edge of that bead appears in $U'$. By Theorem \ref{dsrmk}, $d_1 U =\overline{U'}$, the quotient which collapses each degenerate edge of $U'$ to a vertex.

Let $T^*$ to be the subnecklace of $T$ whose joins include all of the vertices of $T$; we call $T^*$ the \emph{spine} of $T$. The necklace $T^*$ is comprised of the path of edges $[i,i+1]$ between adjacently numbered vertices in each bead of $T$. By Theorem \ref{dsrmk}, $d_0 T = \overline{T^*}$, the quotient which collapses each degenerate edge of $T^*$ to a vertex.

The equation $\overline{T^*} = \overline{U'}$ says that the spine of $T$ equals the diagonal of $U$, modulo any degenerate edges. We construct a simplicial subset $A$ of $X$ that contains the data of $T$ and $U$; by construction any necklace $S$ representing a 2-simplex in $\C X(x,y)$ with $T$ and $U$ as second and zeroth faces will contain $A$. Start with the necklace $U$ and insert degenerate edges between the joins of $U$ if needed until the diagonal of this new necklace is an expansion of the spine of $T$; we might call this new necklace ``stretched''. It won't necessarily equal the spine of $T$ because it may have some extra degenerate edges, which are the diagonals of beads of $U$. To form $A$, we glue the beads of $T$ into this necklace along their spines. This can be done directly whenever the spine of the bead $\s$ of $T$ does not encounter any degenerate edges arising from the diagonal of $U$. If the stretched version of $U$ contains an extra degenerate edge at the location corresponding to the $i$-th vertex of some bead $\s$ of $T$, we replace $\s$ by the degenerate simplex $s_i \s$, which can be glued in as described above. The resulting simplicial set $A$ is obtained by gluing a ``thickened'' version of $T$, where some beads have been expanded to degenerate beads to accommodate any degenerate diagonals of $U$, to the stretched necklace $U$, which has ``gaps'' between beads occurring wherever there are degenerate edges in the spine of $T$.

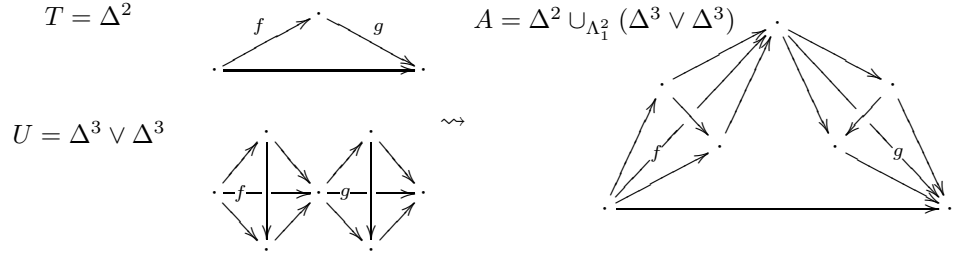
\[\raisebox{.60in}{\xymatrix@R=13pt@C=13pt@M=2pt{ {}\save[]*\txt{$T=\D^2$}\restore & & & \cdot \ar[drr]^g & & \\  & \cdot \ar[rrrr] \ar[urr]^f & & & & \cdot  \\ {}\save[]{U =\D^3 \vee \D^3} \restore& & \cdot \ar[dr] \ar[dd] & & \cdot \ar[dr] \ar[dd] & \\ & \cdot \ar'[r]|f[rr] \ar[ur] \ar[dr] & & \cdot \ar'[r]|g[rr] \ar[ur] \ar[dr] & & \cdot \\ & & \cdot \ar[ur] & & \cdot \ar[ur] & }} \leadsto  \raisebox{.55in}{\xymatrix@=13pt{ {}\save[]*\txt{$A=\D^2\cup_{\L^2_1} (\D^3 \vee \D^3)$}\restore & & & \cdot \ar[dddrrr]|{\hole}|(.70){g} \ar[drr] \ar[ddr] & & & \\ & \cdot \ar[dr] \ar[urr] & & & & \cdot \ar[ddr] \ar[dl] & \\ & & \cdot \ar[uur] & & \cdot \ar[drr] & & \\
\cdot \ar[rrrrrr] \ar[uur] \ar[urr] \ar[uuurrr]|(.30){f}|{\hole} & & & & & & \cdot}}\]

We wish to fill in $A$ to obtain a new necklace $S$ in $X$ on the vertices of $A$, whose joins contain the joins of $T$ and any extra copies of these vertices that appear from degenerate diagonals of $U$.  To construct such a necklace $S$, we need to ``fill'' each simplicial subset between our so-designated joins to a simplex of the appropriate dimension; this constructs the beads one at a time. Inductively, there are only two types of fillings necessary: for one, we must expand two adjacent beads $\D^k$ and $\D^n$ to form a single bead $\D^{k+n}$. For the other we must expand two overlapping beads, where the diagonal of one is the first edge along the spine of another. Such extensions are possible because the maps of \ref{cor1} and \ref{cor2} and its dual are mid anodyne, and $X$ is a quasi-category. This construction defines a totally non-degenerate necklace $S$ in $X$ which contains $A$.

Let $\vec{S}$ be the flag $J_S \subset S^1 \subset V_S$, where $S^1$ is the vertex set of $T$ plus, for each degenerate edge of $U$, an extra copy of the corresponding vertex.  We claim that the triple $(S, k, \vec{S})$ is a filler for the horn $\L^2_1 \ra \C X(x,y)$.  The necklace representing the face $d_2 (S, \vec{S}) = ( \overline{S'}, \overline{d_2\vec{S}})$ is obtained by restricting $S$ to $S^1$, which contains the vertices of $T$ together with some extra copies arising from degenerate edges of $U$. This amounts to restricting $A$ to the same vertices, and so it is clear that the totally non-degenerate quotient $\overline{S'}$ is $T$. Hence, $(\overline{S'}, \overline{d_2\vec{S}}) = (T, \vec{T})$. Similarly, the necklace $S^*$ of $d_0(S, \vec{S}) = (\overline{S^*}, \overline{d_0\vec{S}})$ contains the beads of $U$ but also some degenerate edges between beads. This says exactly that $(\overline{S^*}, \overline{d_0\vec{S}})= (U, \vec{U})$, as desired.
\end{proof}

Recall from the introduction that the 1-simplices in the hom-spaces of $\C X$ can be thought of as ``factorisations,'' which aren't reversible. As a result, we would not expect that $\L^2_0$ or $\L^2_2$ horns in the hom-spaces can be filled in general, even if $X$ is a Kan complex: if these horns could be filled, it would follow that every 1-simplex would be invertible up to homotopy. It is easy to find examples of outer horns that cannot be filled: for any non-degenerate 2-simplex in $\C X(x,y)$, the representatives of the outer faces are proper subnecklaces of the necklace representing the inner face. Exchanging the first face for either the zeroth or the second, one obtains a $\L^2_2$ or $\L^2_0$ horn that has no filler.

Unexpectedly, many higher horns in these hom-spaces can be filled uniquely, without any hypotheses on $X$, by an immediate corollary to the main theorem of the next section.

\section{For any $X$, all $\C X(x,y)$ are 3-coskeletal}\label{cosksec}

Theorems \ref{dsthm} and \ref{dsrmk} can also be used to prove the following.

\begin{thm}\label{coskthm} For any simplicial set $X$ and vertices $x$ and $y$, the hom-space $\C X(x,y)$ is 3-coskeletal.
\end{thm}

Recall, a simplicial set $X$ is $n$-\emph{coskeletal} if any sphere $\partial\Delta^k \ra X$, with $k>n$, can be filled uniquely to a $k$-simplex $\D^k \ra X$, where $\partial\D^k$ is the simplicial subset of $\D^k$ generated by all of its $(k-1)$-dimensional faces (cf. \cite[section 0.7]{duskinsimplicialmethods}). Morally, an $n$-coskeletal Kan complex corresponds to an $(n-1)$-type from classical homotopy theory.

\begin{proof}
We must show that for $n \geq 4$ any sphere $\partial\D^n \ra \C X(x,y)$ has a unique filler. Using Theorem \ref{dsthm}, an $n$-sphere in $\C X(x,y)$ is a collection \[(T_0, \vec{T_0}), \ldots, (T_n, \vec{T_n})\] of totally non-degenerate necklaces in $X$ accompanied by flags $\vec{T_i}$ of vertices \[J_i \subset T_i^1 \subset \cdots \subset T_i^{n-2} \subset V_i\] satisfying the relations \[d_i (T_j, \vec{T_j}) = d_{j-1} (T_i, \vec{T_i}) \hspace{.5cm} \text{for~all}\ i < j.\] By Corollary \ref{dscor} the necklaces $T_i$ are equal for $0<i<n$; we call this common necklace $S$. Furthermore, when $n \geq 4$, the relations between the inner faces define a flag $\vec{S}$ of vertices \[J_S \subset S^1 \subset \cdots \subset S^{n-1} \subset V_S\] such that $d_i \vec{S} = \vec{T_i}$, for $0 < i <n$. 

It is clear that $(S,\vec{S})$ is the only possible filler for this sphere. It remains to show that $d_0(S,\vec{S}) = (T_0, \vec{T_0})$ and $d_n(S, \vec{S}) = (T_n, \vec{T_n})$. Using any inner face $i$, we compute that \begin{equation}\label{eq1} d_{i-1} (T_0, \vec{T_0}) = d_0 (T_i, \vec{T_i}) = d_0 d_i (S, \vec{S}) = d_{i-1} d_0 (S, \vec{S}).\end{equation}
We may choose $i >1$ so that (\ref{eq1}) is an inner face relation; it follows from Corollary \ref{dscor} that $T_0$ is the necklace of $d_0 (S, \vec{S})$ and $d_{i-1}\vec{T_0} = d_{i-1}d_0\vec{S}$. By choosing a different inner face $i>1$, we learn that the omitted sets of the flags $\vec{T_0}$ and $d_0\vec{S}$ also agree, so we conclude that $\vec{T_0} = d_0 \vec{S}$, and hence that $d_0(S,\vec{S}) = (T_0, \vec{T_0})$.

Similarly, using any inner face $j$, we compute that \begin{equation}\label{eq2} d_j(T_n, \vec{T_n}) = d_{n-1} (T_j, \vec{T_j}) = d_{n-1} d_j (S, \vec{S}) = d_j d_n (S, \vec{S}).\end{equation} We may choose $j < n-1$ so that (\ref{eq2}) is an inner face relation; it follows from Corollary \ref{dscor} that $T_n$ is the necklace of $d_n(S, \vec{S})$ and $d_j \vec{T_n} = d_j \overline{d_n\vec{S}}$. We choose any other inner face $j < n-1$ to conclude that the flag $\vec{T_n} = \overline{d_n\vec{S}}$, and hence that $d_n(S,\vec{S}) = (T_n, \vec{T_n})$.
\end{proof}

It is well-known that $n$-coskeletal simplicial sets have unique fillers for all horns of dimension at least $n+2$ (cf., e.g., \cite[section 2.3]{duskinsimplicialmatrices}).

\begin{cor}\label{exactcor} For $n >4$, every horn $\L^n_k \ra \C X(x,y)$ has a unique filler.
\end{cor}
\begin{proof} 
When $n >4$, the map $\sk_3  \L^n_k \ra \sk_3 \D^n$ induced by the inclusion is an isomorphism. Hence, any map $\sk_3 \L^n_k \ra \C X(x,y)$ can be extended uniquely to a map $\sk_3 \D^n \ra \C X(x,y)$. By adjunction, any horn $\L^n_k \ra \cosk_3 \C X(x,y) \cong \C X(x,y)$ has a unique filler.
\end{proof}

\begin{ex} Let $X$ be a simplicial set with two vertices $x$ and $y$, two edges $f$ and $g$ from $x$ to $y$, and a 2-simplex $\a$ and a 3-simplex $\s$ as shown.
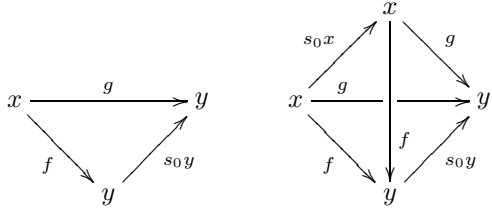
\[\xymatrix{ &  & & & x \ar[dd]^(.7){f} \ar[dr]^g & \\ x \ar[dr]_f \ar[rr]^g & & y & x \ar'[r]^{g}[rr] \ar[dr]_f \ar[ur]^{s_0x} & & y \\ &y \ar[ur]_{s_0y} & & & y \ar[ur]_{s_0y} }\] The necklaces $T_0 = T_3 = \D^2$ mapping onto $\a$ and $T_1 = T_2 = \D^3$ mapping onto $\s$, with flags \begin{align*} &\vec{T_0} & \{0,2\} \subset \{0,1,2\} \subset \{0,1,2\} \\ &\vec{T_1} &\{0,3\} \subset \{0,2,3\} \subset \{0,1,2,3\} \\ &\vec{T_2} & \{0,3\} \subset \{0,1,3\} \subset \{0,1,2,3\} \\ &\vec{T_3} &\{0,2\} \subset \{0,2\} \subset \{0,1,2\} \end{align*} define a 3-sphere in $\C X(x,y)$. This sphere cannot be filled because there is no flag $\vec{S}$ with $d_1 \vec{S} = \vec{T_1}$ and $d_2 \vec{S} = \vec{T_2}$. Hence, $\C X(x,y)$ is not 2-coskeletal, a fact that can also be verified by direct computation of $\C X$.
\end{ex}

This shows that the result of Theorem \ref{coskthm} is the strongest possible.

\section{$\C X$ is not locally quasi, if the quasi-category $X$ is not a category}\label{charsec}

In light of the prior results, one might hope that all inner horns in the hom-spaces of the simplicial category associated to a quasi-category can be filled. However, we will show in this section that for any quasi-category $X$ that is not the nerve of an ordinary category, we can find a hom-space in $\C X$ that is not a quasi-category, proving the following theorem.

\begin{thm}\label{charthm} If $X$ is a quasi-category and its simplicial category $\C X$ is locally quasi, then $X$ is isomorphic to the nerve of a category. 
\end{thm}

More explicitly, we will show that if $X$ is any quasi-category that fails to satisfy either of the following conditions \begin{enumerate} \item $X$ is 2-coskeletal \item $X$ has unique fillers for $\L^2_1$, $\L^3_1$, and $\L^3_2$ horns\end{enumerate} then there is a $\L^3_1$ horn in some hom-space of $\C X$ that cannot be filled. By a lemma below, any quasi-category that satisfies both of these conditions is isomorphic to the nerve of a category. Thus, if $X$ is a quasi-category such that the simplicial category $\C X$ is locally quasi, then $X$ must be the nerve of a category. In the following section, we complete our characterization of the quasi-categories whose simplicial categories are locally quasi, by considering the case where $X$ is isomorphic to the nerve of a category.\footnote{Note that the simplicial category $\C\L^2_1$ is locally quasi, its three non-empty hom-spaces being trivial, but $\L^2_1$ is not itself a quasi-category. We don't concern ourselves with the technicalities of determining which non-quasi-categories have locally quasi simplicial categories as a fluke.} 

It is well-known that a simplicial set is isomorphic to the nerve of a category if and only if every inner horn has a unique filler. A third equivalent condition is given in the following lemma, though we only prove the implication needed here. For a full proof see \cite[proposition 1.13]{joyaltheoryI}.

\begin{lem}\label{catlem} Let $X$ be a 2-coskeletal simplicial set such that for $n=2$ or $3$ every inner horn $\L^n_k \ra X$ has a unique filler. Then $X$ is isomorphic to the nerve of a category.
\end{lem}
\begin{proof}
By the argument given for Corollary \ref{exactcor}, any horn $\L^n_k \ra X$ with $n>3$ has a unique filler.
\end{proof}

The following lemmas give explicit $\L^3_1$ horns that cannot be filled.

\begin{lem}\label{badexlem} Let  $X$ be a quasi-category with distinct $n$-simplices $\s$ and $\t$ with the same boundary, for some $n \geq 3$. Let $x$ and $y$ be the common initial and terminal vertices, respectively, of these simplices. Then there is a horn $\L^3_1 \ra \C X(x,y)$ with no filler.
\end{lem}
\begin{proof}
Degenerate simplices with given boundaries are unique, so at least one of $\s$ or $\t$ is non-degenerate; assume it is $\t$. Let $S$ and $T$ equal $\D^n$ with maps $S \ra X$ and $T \ra Y$ that send the unique bead to $\s$ and $\t$, respectively. Pick some proper subset $J$ of vertices of $\D^n$ that contains both the initial and final vertices and at least one other. Let $U$ be the maximal subnecklace of $\D^n$ that has these vertices as joins; i.e., let $U$ be the $J$-splitting of $\D^n$. There is a natural map $U \ra X$ that factors through both $S$ and $T$ because the image of $U$ in $X$ sits inside the common boundary of $\s$ and $\t$.

We claim that $U$, $S$, and $T$ form, respectively, the zeroth, second, and third faces of a horn $\L^3_1 \ra \C X(x,y)$ when given flags \begin{align*} &\vec{U} & J \subset [n] \subset [n] \\ &\vec{S} & \{0,n\} \subset J \subset [n] \\ & \vec{T} & \{0,n\} \subset J \subset [n] \end{align*} where we replace $(U, \vec{U})$ and $(S, \vec{S})$ by their totally non-degenerate quotients if necessary; we ignore this possibility in our notation as it does not change the argument in any substantial way. To prove this, we must check that $d_0 S = d_1 U$, $d_0 T = d_2 U$, and $d_2 S = d_2 T$. The first two equations say that the $J$-splittings of $S$ and $T$ are $U$, which is true by definition. The final equation says that $S$ and $T$ are the same when restricted to $J$, which is true because the image of this restriction lies in the common boundary of $\s$ and $\t$.

A filler of this horn would necessarily have $(S, \vec{S})$ as its second face. By Corollary \ref{dscor}, any such 3-simplex must be given by the necklace $S$ itself together with the flag \[\{0,n\} \subset J \subset S^2 \subset [n].\] But the third face any such 3-simplex can't possibly be represented by the necklace $T$. Hence, this horn has no filler.
\end{proof}

\begin{lem}\label{lowdimlem} Let $X$ be a quasi-category with distinct 3-simplices $\s$ and $\t$ such that $d_0 \s = d_0 \t$ and $d_2 \s = d_2 \t$. Let $x$ and  $y$ be the common initial and terminal vertices, respectively, of these simplices. Then there is a horn $\L^3_1 \ra \C X(x,y)$ with no filler.
\end{lem}
\begin{proof}
The argument given in the above proof with $n=3$ and $J = \{0,1, 3\}$ applies in this case. The conditions that $U$, $S$, and $T$ form a $\L^3_1$ horn amount to the requirement that the simplices $\s$ and $\t$ in the images of $S$ and $T$ respectively share common zeroth and second faces.
\end{proof}

\begin{proof}[Proof of Theorem \ref{charthm}] If $X$ is not 2-coskeletal, then either there exist distinct $n$-simplices with the same boundary, in which case we may apply Lemma \ref{badexlem}, or there exists an $n$-sphere in $X$ that cannot be filled for some $n>2$. In the latter case, choose $0 < i < n$ and fill the $\L^n_i$ horn contained in the $n$-sphere to obtain, as the $i$-th face of the filler, and ($n-1$)-simplex, necessarily distinct from the one appearing in the sphere, but with the same boundary. Unless $n$ was equal to 3, we have found distinct simplices of dimension 3 or higher with the same boundary, and we can apply Lemma \ref{badexlem} to show that $\C X$ is not locally quasi. 

In the remaining case, we have two distinct 2-simplices $\a$ and $\b$ with common boundary. We construct the 3-simplices $\s$ and $\t$ of Lemma \ref{lowdimlem} by filling two $\L^3_1$-horns with second face $\a$, third face either $\a$ or $\b$, and zeroth face the appropriate degeneracy. Applying Lemma \ref{lowdimlem}, we conclude that $\C X$ is not locally quasi if $X$ is not 2-coskeletal.

Alternatively, if $X$ has some $\L^3_2$ or $\L^3_1$ horn with two distinct fillers, it is clear that Lemma \ref{lowdimlem} or the dual argument can be applied. If $X$ has some $\L^2_1$ horn with distinct fillers $\a$ and $\b$, we repeat the construction just given, noting that it does not matter if $d_1 \a \neq d_1 \b$. Thus, if $X$ does not have unique fillers for low dimensional inner horns, then $\C X$ is not locally quasi.
\end{proof}

\section{Simplicial categories associated to categories}\label{catsec}

It remains to consider the simplicial categories $\C N\A$ associated to (ordinary) nerves of categories $\A$; $\C N\A$ is often referred to as the ``simplicial thickening'' of $\A$. More specifically, $\C N\A$ is the cofibrant replacement of $\A$, regarded as a trivial simplicial category, in the usual model structure for simplicial categories \cite{bergnersimplicialcategories}. In this section, we show that the hom-spaces of $\C X$ are 2-coskeletal when $X$ is isomorphic to the nerve of a category. We then provide examples of $\L^3_1$ and $\L^3_2$ horns in some hom-spaces that cannot be filled, whenever the category has a non-trivial factorization of an identity morphism. Finally, we show that the simplicial category $\C N\A$ is isomorphic to another known cofibrant replacement of $\A$ in $\sCat$: namely, the standard free simplicial resolution of Section \ref{intsec}. 

Throughout this section, we assume that $X$ is isomorphic to the nerve of a category. We begin by noting an obvious lemma with a useful corollary.

\begin{lem} A simplex in $X$ is degenerate if and only if its spine contains a degenerate edge.
\end{lem}

\begin{cor}\label{nicecor} If $T \ra X$ is a totally non-degenerate necklace and $J_T \subset K \subset V_T$ is a collection of vertices of $T$ containing the joins, then the $K$-splitting of $T$ is totally non-degenerate.
\end{cor}
\begin{proof} The spine of $T$ equals the spine of its $K$-splitting.
\end{proof}

\begin{rmk}\label{nicermk} In particular, if $(T, \vec{T})$ is totally non-degenerate, its zeroth face is the $T^1$-splitting of $T$ with flag $d_0 \vec{T}$. In practice this that we can recover much of the data of the flag $\vec{T}$ from its zeroth face, as well as from the inner faces.
\end{rmk}

\begin{thm}\label{catcoskthm} For any simplicial set $X$ that is the nerve of a category and any objects $x$ and $y$, the hom-space $\C X(x,y)$ is 2-coskeletal.
\end{thm}
\begin{proof}
By Theorem \ref{coskthm}, it remains to show that any sphere $\partial \D^3 \ra \C X(x,y)$ has a unique filler. Recall, a $3$-sphere in $\C X(x,y)$ is a collection \[(T_0, \vec{T_0}), \ldots, (T_3, \vec{T_3})\] of totally non-degenerate necklaces in $X$ accompanied by flags $\vec{T_i}$ of vertices \[J_i \subset T_i^1 \subset  V_i\] satisfying the relations \[d_i (T_j, \vec{T_j}) = d_{j-1} (T_i, \vec{T_i}) \hspace{.5cm} \text{for~all}\ i < j.\] 

Corollary \ref{dscor} and the relation between the inner faces implies that the necklaces $T_1$ and $T_2$ are equal; we call this common necklace $S$. By Remark \ref{nicermk}, the relations between the zeroth, first, and second faces define a flag $\vec{S}$ of vertices \[J_1 = J_2\ \subset\ T^1_2 = J_0\ \subset\ T^1_0 = T^1_1\ \subset\ V_0 = V_1 = V_2\] such that $d_i \vec{S} = \vec{T_i}$ for $i < 3$. It is clear that $(S,\vec{S})$ is the only possible filler for this sphere. The relation (\ref{eq1}) with $i=2$ implies that $T_0$ is the necklace of $d_0(S, \vec{S})$; hence, $d_0(S, \vec{S})=(T_0, \vec{T_0})$. The relation (\ref{eq2}) with $j=1$ implies that $T_3$ is the necklace of $d_3(S,\vec{S})$. By Remark \ref{nicermk}, we may use $j=0$ in (\ref{eq2}) to conclude that $\vec{T_n} = \overline{d_3\vec{S}}$, and hence that $d_3 (S, \vec{S}) = (T_3, \vec{T_3})$.
\end{proof}

\begin{ex} Suppose $X$ is the nerve of a category with non-identity morphisms $s \colon x \ra y$ and $r \colon y \ra x$ such that $rs$ is the identity at $x$. Let $T$ be the necklace $\D^3$ and let $U$ be the necklace $\D^2 \vee \D^1$ whose images in $X$ have spines $srs$. Then there is a 2-simplex $\a$ in $\C X(x,y)$ with zeroth face $U$, first face $T$, and second face a degeneracy, but there is no 2-simplex with the positions of $U$ and $T$ reversed. The simplex $\a$ can be glued to degenerate simplices to form horns, depicted below, that have no filler.
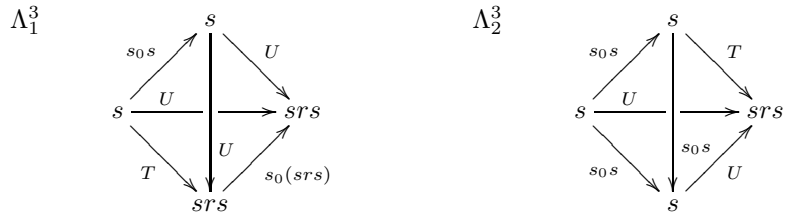
\[\xymatrix@=35pt@!0{ \L^3_1 & &  s \ar[dd]^(.7){U} \ar[dr]^{U} & & & \L^3_2 &  & s \ar[dd]^(.7){s_0 s} \ar[dr]^T & \\ &  s \ar[ur]^{s_0s} \ar[dr]_T \ar'[r]^-{U}[rr]  & & srs & & & s \ar'[r]^-{U}[rr] \ar[dr]_{s_0s} \ar[ur]^{s_0s} & & srs \\ &&  srs \ar[ur]_{s_0(srs)} & & &&  & s \ar[ur]_{U} }\]
\end{ex}

\begin{rmk} When $X$ is isomorphic to the nerve of a category such that identity morphisms cannot be factored, or equivalently, so that there do not exist $r$ and $s$ as above, then any restriction of a totally non-degenerate necklace in $X$ is totally non-degenerate. By arguments similar to those given above, all $\L^2_1$, $\L^3_1$, and $\L^3_2$ horns in hom-spaces of $\C X$ can be filled uniquely, and it follows from Lemma \ref{catlem}, that the $\C X(x,y)$ are themselves nerves of categories in this case. The poset categories $[n]$ do satisfy this condition, but most interesting examples do not.
\end{rmk}

We conclude with a theorem that is most likely known somewhere, given the ubiquity of the simplicial resolution construction described in the introduction, but which, with Theorem \ref{dsthm}, admits a particularly simple proof.

\begin{thm} For any category $\A$, the simplicial category $\C NA$ is isomorphic to the simplicial category obtained as the standard free simplicial resolution of $\A$.
\end{thm}
\begin{proof}
The objects of both simplicial categories are the objects of $\A$. It remains to show that the hom-spaces coincide. A necklace in the nerve of a category is uniquely determined by its spine and the set of joins; i.e., a necklace is a sequence of composable non-identity morphisms each contained in one set of parentheses, indicating which morphisms are grouped together to form a bead. An $n$-simplex in a hom-object of the standard free simplicial resolution is a sequence of composable non-identity morphisms, each contained within $(n-1)$ sets of parentheses. The morphisms in the sequence describe the spine of a necklace and the locations of each level of parentheses defines the corresponding set in the flag of vertex data; by Theorem \ref{dsthm}, this exactly specifies an $n$-simplex in the corresponding hom-object of $\C NA$.
\end{proof}

\begin{acknowledgments}
The author would like to thank Dominic Verity for enduring several conversations on this topic and being a very generous host. She would also like to thank her advisor, Peter May, who suggested the title for this paper, and the anonymous reviewer, who suggested several ways to improve its exposition. The author is grateful for support from the National Science Foundation, whose Graduate Research Fellowship allowed her to visit Macquarie University, where this work took place.
\end{acknowledgments}

\end{document}